\documentclass [12pt] {article}
\usepackage{amsfonts}

\newtheorem{theorem}{Theorem}

\newtheorem{corollary}[theorem]{Corollary}

\newtheorem{lemma}[theorem]{Lemma}

\newenvironment{proof}[1][Proof]{\noindent\textbf{#1.} }{\ \rule{0.5em}{0.5em}}
\begin{document}

\title{Concentration of random polytopes around the expected convex hull}
\author{Daniel J. Fresen\thanks{Yale University, Department of Mathematics, daniel.fresen@yale.edu} \hspace{0in} and  Richard A. Vitale \thanks{The University of Connecticut, Department of Statistics, r.vitale@uconn.edu}}
\maketitle

\begin{abstract}
We provide a streamlined proof and improved estimates for the weak
multivariate Gnedenko law of large numbers on concentration of random
polytopes within the space of convex bodies (in a fixed or a high
dimensional setting), as well as a corresponding strong law of large numbers.
\end{abstract}

\section{Introduction}

Let $d\in \mathbb{N}$ and let $\mu $ be a probability measure on $\mathbb{R}%
^{d}$ with a log-concave density $f=d\mu /dx$, i.e. $-\log f$ is a convex
extended real valued function. Let $n\geq d+1$ and let $(X_{i})_{1}^{n}$
denote an i.i.d. sequence of random vectors with common distribution $\mu $.
The convex hull%
\begin{equation}
P_{n}=\mathrm{conv}\{X_{i}\}_{1}^{n}  \label{random poly def}
\end{equation}%
is a random polytope and, as such, is a random element w.p.1 of the space $%
\mathcal{K}_{d}$ of all convex bodies in $\mathbb{R}^{d}$ (compact convex
sets with non-empty interior). There are various metrics and metric-like
functions on $\mathcal{K}_{d}$, such as the Hausdorff distance $d_{\mathcal{H%
}}$ and the Banach-Mazur distance $\delta ^{BM}$ (for origin symmetric
bodies). We refer the reader to \cite{Schn} for general background on convex
bodies, and to \cite{Gruber} specifically for metric, and other, structures
on $\mathcal{K}^{d}$.

It was shown in \cite{Fr2013} that if $n\geq c\exp (\exp (5d))$, then with
probability at least $1-3^{d+3}(\log n)^{-1000}$, there exists $x\in \mathbb{%
R}^{n}$ and%
\[
\lambda \leq 1+c^{\prime }d^{2}\frac{\log \log n}{\log n}
\]%
such that%
\begin{equation}
\lambda ^{-1}(F_{1/n}-x)+x\subseteq P_{n}\subseteq \lambda (F_{1/n}-x)+x
\label{sandwhiching}
\end{equation}%
where $c,c^{\prime }>0$ are universal constants and $F_{1/n}$ is the
floating body defined by%
\begin{equation}
F_{\delta }=\cap \{\mathfrak{H}:\mu (\mathfrak{H})\geq 1-\delta \}
\label{floatin}
\end{equation}%
where the intersection runs through the collection of all closed half-spaces 
$\mathfrak{H}$ of $\mu $-mass at least $1-\delta $ ($\delta <e^{-1}$). The
body $F_{1/n}$ was originally defined by Sch\"{u}tt and Werner \cite{ScWe}
in the case of Lebesgue measure on a convex body and has often been used to
model random polytopes, see for example \cite{Bar07, BaLa, Vu}.

Being log-concave, the density $f$ decays at least as quickly as an
exponential function. Any bound on the decay rate of $f$ translates to a
bound on the Hausdorff distance $d_{\mathcal{H}}(P_{n},F_{1/n})$. For
example if the tails of $\mu $ are sub-Gaussian (with universally bounded
constants), then $\mathrm{diam}(F_{1/n})\leq c(\log n)^{1/2}$ and (\ref%
{sandwhiching}) translates to%
\[
d_{\mathcal{H}}(P_{n},F_{1/n})\leq c^{\prime }d^{2}\frac{\log \log n}{\sqrt{%
\log n}} 
\]%
where $c,c^{\prime }>0$ are universal constants. This is an embodiment of
the concentration of measure phenomenon: the polytope $P_{n}$, as a random
element of the metric space $(\mathcal{K}_{d},d_{\mathcal{H}})$, is
concentrated around $F_{1/n}$.

In the case $d=1$, $P_{n}$ reduces to the interval%
\[
\lbrack \min \{X_{i}\}_{1}^{n},\max \{X_{i}\}_{1}^{n}] 
\]%
and we see that the above mentioned result generalizes a theorem of Gnedenko 
\cite{Gn} on concentration of the maximum and minimum of a large i.i.d.
sample (under rapid decay of the tails of $\mu $). Other multivariate
analogs of Gnedenko's law of large numbers are included in \cite{Geff} for
the multivariate normal distribution, \cite{Good} for Gaussian measures on
infinite dimensional spaces, \cite{DMR, Fish 1966, Fish 1969} for regularly
varying distributions, and \cite{KiRe, MR94} for more general distributions.

The proof of (\ref{sandwhiching}) was complicated by the fact that there is
no convenient expression for the support function of the floating body,%
\[
h_{F_{1/n}}(\theta )=\max_{x\in F_{1/n}}\left\langle \theta ,x\right\rangle 
\]%
In this paper we study concentration of $P_{n}$ around the expected convex
hull%
\begin{equation}
\mathbb{E}P_{n}=\{x\in \mathbb{R}^{n}:\forall \theta \in
S^{d-1},\left\langle \theta ,x\right\rangle \leq \mathbb{E}\max_{1\leq i\leq
n}\left\langle \theta ,X_{i}\right\rangle \}  \label{def expected hull}
\end{equation}%
which is easily seen to be a convex body with support function%
\begin{equation}
h_{\mathbb{E}P_{n}}(\theta )=\mathbb{E}\max_{1\leq i\leq n}\left\langle
\theta ,X_{i}\right\rangle   \label{support expected}
\end{equation}%
Using the expected convex hull leads to a streamlined proof of (\ref%
{sandwhiching}). The notion of the expectation of a random convex body
follows the theory of integrals of set valued functions, see for example 
\cite{Art, Au, Deb, Kudo} and the references therein. It was used in \cite%
{ArtVit} for the purpose of a Kolmogorov strong law of large numbers and has
appeared as an approximant to floating bodies in bounded domains \cite{BaVi}%
, as well as in other contexts e.g. \cite{GW92, GW12, Vitale87, Vitale 90,
Vitale 91, Vitale 94, Weil95}.

In the original paper \cite{Fr2013} we were mainly interested in a
quantitative dependence on $n$. Although our bounds included dependence on
dimension, the required sample size was very large. Theorem \ref{polytope
and expected val} includes improved bounds on the required sample size and
is more in the spirit of the high dimensional theory. The quantitative
dependence that we achieve is essentially the same as that in Dvoretzky's
theorem, see for example \cite{Sch}. This result should also be compared to
the main result in \cite{DGT}.

To make the present exposition brief, we refer the reader to \cite{Fr2013}
for a more detailed discussion.

\section{Main results}

\begin{theorem}
\label{polytope and expected val}Let $d\in \mathbb{N}$ and let $\mu $ be a
log-concave probability measure on $\mathbb{R}^{d}$ with center of mass at
the origin and non-singular covariance matrix. Consider any $\varepsilon \in
(0,1/2)$ and let $n\geq \exp (7d\varepsilon ^{-1}\log \varepsilon ^{-1})$.
Let $(X_{i})_{1}^{n}$ be an i.i.d. sample from $\mu $, $P_{n}=\mathrm{conv}%
\{X_{i}\}_{1}^{n}$, and let $\mathbb{E}P_{n}$ denote the expected convex
hull as defined by (\ref{support expected}). With probability at least
\thinspace $1-3n^{-\varepsilon /4}$,%
\[
(1-\varepsilon )\mathbb{E}P_{n}\subseteq P_{n}\subseteq (1+\varepsilon )%
\mathbb{E}P_{n} 
\]
\end{theorem}

Using the bound $d_{\mathcal{H}}(A,B)\leq \mathrm{diam}(B)\inf \{\lambda
\geq 1:\lambda ^{-1}A\subseteq B\subseteq \lambda A\}$, Theorem \ref%
{polytope and expected val} may be transferred to a bound on $d_{\mathcal{H}%
}(P_{n},\mathbb{E}P_{n})$. The following Corollary, which is similar to the
main result in \cite{BaVi}, is a consequence of Lemma \ref{2 sided 1D
concentration}.

\begin{corollary}
\label{expected val and floating body}Let $d\in \mathbb{N}$ and let $\mu $
be a log-concave probability measure on $\mathbb{R}^{d}$ with center of mass
at the origin and non-singular covariance matrix. Let $\mathbb{E}P_{n}$
denote the expected convex hull as defined by (\ref{support expected}), and
let $F_{1/n}$ denote the floating body defined by (\ref{floatin}). Then
provided $n\geq 12$,%
\[
(1-3/\log n)\mathbb{E}P_{n}\subseteq F_{1/n}\subseteq (1+1/\log n)\mathbb{E}%
P_{n} 
\]
\end{corollary}

\begin{theorem}
\label{strong LLN}Let $d\in \mathbb{N}$ and let $\mu $ be a log-concave
probability measure on $\mathbb{R}^{d}$ with center of mass at the origin
and non-singular covariance matrix. Let $(X_{i})_{1}^{\infty }$ be an i.i.d.
sample from $\mu $, and let $(P_{n})_{n=d+1}^{\infty }$ and $(\mathbb{E}%
P_{n})_{3}^{\infty }$ be the random polytopes and expected convex hulls
defined by (\ref{random poly def}) and (\ref{support expected})
respectively. Then with probability $1$, there exists $N\in \mathbb{N}$ such
that for all $n\geq N$,%
\begin{equation}
\left( 1-\frac{3\log \log n}{\log n}\right) \mathbb{E}P_{n}\subseteq
P_{n}\subseteq \left( 1+\frac{8\log \log n}{\log n}\right) \mathbb{E}P_{n}
\label{strong bound}
\end{equation}
\end{theorem}

\section{Notation}

If $J$ is the cumulative distribution function associated to a probability
measure $\mu $ on $\mathbb{R}$, then the generalized inverse $%
J^{-1}:(0,1)\rightarrow \mathbb{R}$ is defined as%
\[
J^{-1}(t)=\sup \{x\in \mathbb{R}:J(x)<t\}=\inf \{x\in \mathbb{R}:J(x)\geq
t\} 
\]%
If $\mu $ has a log-concave density function then $J(J^{-1}(t))=t$ for all $%
t\in (0,1)$ and $J^{-1}(J(x))=x$ for all $x$ in the support of $\mu $. If $%
(Y_{i})_{1}^{n}$ is an i.i.d. sample from $\mu $, then $Y_{(n)}=\max_{1\leq
i\leq n}Y_{i}$ denotes the $n^{th}$ order statistic.

If $K\subset \mathbb{R}^{d}$ is a convex body then the function%
\[
h_{K}(x)=\max_{y\in K}\left\langle x,y\right\rangle 
\]%
is known as the support function of $K$. If $0\in \mathrm{int}(K)$ then the
Minkowski functional is defined as 
\[
\left\Vert x\right\Vert _{K}=\min \{\lambda \geq 1:x\in \lambda K\} 
\]%
and the support function is the Minkowski functional of the polar body%
\[
K^{\circ }=\{y\in \mathbb{R}^{d}:\forall x\in K,\left\langle
x,y\right\rangle \leq 1\} 
\]%
i.e. $h_{K}(\cdot )=\left\Vert \cdot \right\Vert _{K^{\circ }}$. In the case
when $K$ is centrally symmetric, i.e. $K=-K$, then $h_{K}(\cdot )$ and $%
\left\Vert \cdot \right\Vert _{K}$ are norms.

\section{Proofs}

The following lemma is a natural extension of Lemma 7 in \cite{Fr2013}.

\begin{lemma}
\label{2 sided 1D concentration}Let $\mu $ be a probability measure on $%
\mathbb{R}$ with mean $0$ and log-concave density $f=d\mu /dx$. Let $n\geq
12 $ and let $(Y_{i})_{1}^{n}$ be an i.i.d. sample from $\mu $. Then for all 
$t>0$,%
\begin{eqnarray}
\mathbb{P}\{Y_{(n)} &\leq &(1+t)\mathbb{E}Y_{(n)}\}\geq 1-n^{-t/2}
\label{right tail} \\
\mathbb{P}\{Y_{(n)} &\geq &(1-t)\mathbb{E}Y_{(n)}\}\geq 1-\exp (-n^{t/2}/3)
\label{left tail}
\end{eqnarray}
\end{lemma}

\begin{proof}
Let $J$ be the common distribution function of each $Y_{i}$. Let $f_{n}$ and 
$J_{n}$ denote the density and distribution function of $Y_{(n)}$,%
\begin{eqnarray*}
J_{n}(t) &=&J(t)^{n} \\
f_{n}(t) &=&\frac{d}{dt}J_{n}(t)=nJ(t)^{n-1}f(t)
\end{eqnarray*}%
Since $f$ is log-concave, so is $J$ (see for example Theorem 5.1 in \cite%
{LoVe} or Lemma 5 in \cite{Fr2013}). The product of log-concave functions is
certainly log-concave, and therefore so is $f_{n}$. By a standard result,
see for example Lemma 5.4 in \cite{LoVe}, $J_{n}^{-1}(e^{-1})\leq \mathbb{E}%
Y_{(n)}\leq J_{n}^{-1}(1-e^{-1})$. Just as the left tail $J$ is log-concave,
so is the right tail $1-J$, and the function $u(t)=-\log (1-J(t))$ is
convex. This implies that,%
\[
\frac{u(J^{-1}(1-n^{-t/2}/n))-u(J^{-1}(1-1/n))}{%
J^{-1}(1-n^{-t/2}/n)-J^{-1}(1-1/n)}\geq \frac{u(J^{-1}(1-1/n))-u(0))}{%
J^{-1}(1-1/n)} 
\]%
which translates to%
\[
\frac{J^{-1}(1-n^{-t/2}/n)-J^{-1}(1-1/n)}{J^{-1}(1-1/n)}\leq \frac{t\log n}{%
2(\log n-1)}\leq t 
\]%
Now,%
\[
\mathbb{P}\{Y_{(n)}\leq J^{-1}(1-n^{-t/2}/n)\}=(1-n^{-t/2}/n)^{n}\geq
1-n^{-t/2} 
\]%
By definition of $J_{n}$, $J_{n}(J^{-1}(1-1/n))=(1-1/n)^{n}<e^{-1}$, so $%
\mathbb{E}Y_{(n)}\geq J_{n}^{-1}(e^{-1})\geq J^{-1}(1-1/n)$ and (\ref{right
tail}) follows. Again by convexity of $u$, 
\[
\frac{u(J^{-1}(1-9/(20n)))-u(J^{-1}(1-9n^{t/2-1}/20))}{%
J^{-1}(1-9/(20n))-J^{-1}(1-9n^{t/2-1}/20)}\geq \frac{%
u(J^{-1}(1-9/(20n)))-u(0))}{J^{-1}(1-9/(20n))} 
\]%
which translates to%
\[
\frac{J^{-1}(1-9/(20n))-J^{-1}(1-9n^{t/2-1}/20)}{J^{-1}(1-9/(20n))}\leq 
\frac{(t/2)\log n}{\log n-1+\log (20/9)}\leq t 
\]%
Now,%
\[
\mathbb{P}\{Y_{(n)}\leq J^{-1}(1-9n^{t/2-1}/20)\}=(1-9n^{t/2-1}/20)^{n}\leq
\exp (-9n^{t/2}/20) 
\]%
As before, $J_{n}(J^{-1}(1-9/(20n)))=(1-9/(20n))^{n}>1-e^{-1}$, so $\mathbb{E%
}Y_{(n)}\leq J_{n}^{-1}(1-e^{-1})<J^{-1}(1-9/(20n))$ and (\ref{left tail})
follows.
\end{proof}

\begin{proof}[Proof of Corollary \protect\ref{expected val and floating body}%
]
Since $J^{-1}(1-1/n)=J_{n}^{-1}((1-1/n)^{n})$, where $J_{n}(x)=\mathbb{P}%
\{Y_{(n)}\leq x\}$, $\mathbb{P}\{Y_{(n)}\leq J^{-1}(1-1/n)\}\geq 1/3$ and by
inequality (\ref{left tail}) of Lemma \ref{2 sided 1D concentration}, this
can only be true if $J^{-1}(1-1/n)\geq (1-(\log 18)/\log n)\mathbb{E}Y_{(n)}$%
. By similar reasoning, $\mathbb{P}\{Y_{(n)}>J^{-1}(1-1/n)\}\geq 1-e^{-1}$,
which by inequality (\ref{right tail}) of Lemma \ref{2 sided 1D
concentration} implies that $J^{-1}(1-1/n)\leq (1+1/\log n)\mathbb{E}\gamma
_{(n)}$. The result now follows from the definitions of $F_{1/n}$ and $%
\mathbb{E}P_{n}$, see (\ref{floatin}) and (\ref{support expected}).
\end{proof}

The following lemma appears as Lemmas 4.10 and 4.11 in \cite{Pi} under the
assumption that $K$ is centrally symmetric. We sketch the proof to show that
it can also be used in the non-symmetric case.

\begin{lemma}
\label{net lemma}Let $K\subset \mathbb{R}^{d}$ be any convex body with $0\in 
\mathrm{int}(K)$ and $0<\varepsilon <1/2$. Then there exists a set $\mathcal{%
N}\subset \partial K$ with $|\mathcal{N}|\leq (3/\varepsilon )^{d}$ such
that for all $\theta \in \partial K$ there exist sequences $(\omega
_{i})_{0}^{\infty }\subseteq \mathcal{N}$ and $(\varepsilon
_{i})_{1}^{\infty }\subseteq \lbrack 0,\infty )$ such that $0\leq
\varepsilon _{i}\leq \varepsilon ^{i}$ for all $i$ and%
\[
\theta =\omega _{0}+\sum_{i=1}^{\infty }\varepsilon _{i}\omega _{i} 
\]
\end{lemma}

\begin{proof}
Consider a subset $\mathcal{N}\subset \partial K$, minimal with respect to
set inclusion, with the following property: for all $z\in \partial K$ there
exists $\omega \in \mathcal{N}$ such that $\left\Vert z-\omega \right\Vert
_{K}\leq \varepsilon $. Such a set can easily be constructed recursively,
and we shall refer to $\mathcal{N}$ as an $\varepsilon $-net. Note that
since $K$ may be non-symmetric, we may have $\left\Vert z-\omega \right\Vert
_{K}\neq \left\Vert \omega -z\right\Vert _{K}$ and order becomes important.
By the standard volumetric argument $|\mathcal{N}|\leq (3/\varepsilon )^{d}$%
. By the defining property of $\mathcal{N}$, for all $x\in \mathbb{R}^{d}$
there exists $\omega \in \mathcal{N}$ such that%
\begin{equation}
\left\Vert x-\left\Vert x\right\Vert _{K}\omega \right\Vert _{K}\leq
\varepsilon \left\Vert x\right\Vert _{K}  \label{net approxi}
\end{equation}%
Now consider $\theta \in \partial K$. By (\ref{net approxi}) there exists $%
\omega _{0}\in \mathcal{N}$ such that $\left\Vert \theta -\omega
_{0}\right\Vert _{K}\leq \varepsilon $. By applying (\ref{net approxi})
again, there exists $\omega _{1}\in \mathcal{N}$ such that $\left\Vert
\theta -\omega _{0}-\left\Vert \theta -\omega _{0}\right\Vert _{K}\omega
_{1}\right\Vert _{K}\leq \varepsilon \left\Vert \theta -\omega
_{0}\right\Vert _{K}\leq \varepsilon ^{2}$. Iterating this procedure defines
a sequence $(\omega _{i})_{0}^{\infty }$ such that for all $N\in \mathbb{N}$,%
\[
\left\Vert \theta -\omega _{0}-\sum_{i=1}^{N}\varepsilon _{i}\omega
_{i}\right\Vert _{K}\leq \varepsilon ^{N+1} 
\]%
where $\varepsilon _{i}=\left\Vert \theta -\omega
_{0}-\sum_{i=1}^{i-1}\varepsilon _{i}\omega _{i}\right\Vert _{K}\leq
\varepsilon ^{i}$.
\end{proof}

\begin{proof}[Proof of Theorem \protect\ref{polytope and expected val}]
Set $\delta =3n^{-\varepsilon /(4d)}$ and let $\mathcal{N\subset \partial ((}%
\mathbb{E}P_{n})^{\circ })$ be a $\delta $-net as in Lemma \ref{net lemma}.
By the bounds imposed on $n$, $\delta \leq \varepsilon /5<1/10$. From the
union bound and Lemma \ref{2 sided 1D concentration}, the following event
occurs with probability at least $1-(3/\delta )^{d}3n^{-\varepsilon /2}\geq
1-3n^{-\varepsilon /4}$: for all $\omega \in \mathcal{N}$,%
\begin{equation}
(1-\varepsilon /2)\left\Vert \omega \right\Vert _{(\mathbb{E}P_{n})^{\circ
}}\leq \left\Vert \omega \right\Vert _{P_{n}{}^{\circ }}\leq (1+\varepsilon
/2)\left\Vert \omega \right\Vert _{(\mathbb{E}P_{n})^{\circ }}
\label{first bound}
\end{equation}%
For any $\theta \in \mathcal{\partial ((}\mathbb{E}P_{n})^{\circ })$, write $%
\theta =\omega _{0}+\sum_{1}^{\infty }\delta _{i}\omega _{i}$, with $\omega
_{i}\in \mathcal{N}$ and $0\leq \delta _{i}\leq \delta ^{i}$ for all $i$. By
the triangle inequality and (\ref{first bound}),%
\[
\left\Vert \theta \right\Vert _{P_{n}{}^{\circ }}\leq (1+\varepsilon
/2)\sum_{i=0}^{\infty }\delta ^{i}\leq (1+2\delta )(1+\varepsilon /2)\leq
1+\varepsilon 
\]%
and%
\[
\left\Vert \theta \right\Vert _{P_{n}{}^{\circ }}\geq \left\Vert \omega
_{0}\right\Vert _{P_{n}{}^{\circ }}-\sum_{1}^{\infty }\delta ^{i}\left\Vert
\omega _{i}\right\Vert _{P_{n}{}^{\circ }}\geq 1-\varepsilon
/2-(1+\varepsilon /2)\delta (1-\delta )^{-1}\geq 1-\varepsilon 
\]%
and the result follows.
\end{proof}

\begin{proof}[Proof of Theorem \protect\ref{strong LLN}]
Here $d$ and $\mu $ are fixed, and we treat $n\rightarrow \infty $ as a
variable. From comparing successive terms in the binomial theorem and using
the fact that $n^{-k}{n\choose k}$ is a decreasing function of $k$, for all $%
\delta \in (0,1/2)$%
\[
(1-2\delta /n)^{n}=(1-\delta )-\delta +{n\choose 2}\left( \frac{2\delta }{n}%
\right) ^{2}+\sum_{k=3}^{n}(-1)^{k}{n\choose k}\left( \frac{2\delta }{n}%
\right) ^{k}\leq 1-\delta 
\]%
Since $1-3n^{-\varepsilon /4}\leq \mathbb{P}\{P_{n}\subseteq (1+\varepsilon )%
\mathbb{E}P_{n}\}=(\mathbb{P}\{X_{1}\in (1+\varepsilon )\mathbb{E}%
P_{n}\})^{n}$, it follows that $\mu ((1+\varepsilon )\mathbb{E}P_{n})\geq
(1-3n^{-\varepsilon /4})^{1/n}\geq 1-6n^{-1-\varepsilon /4}$ (provided $%
3n^{-\varepsilon /4}<1/2$). Setting $\varepsilon =8(\log \log n)/\log n$
yields%
\[
\sum_{n=12}^{\infty }\mathbb{P}\{X_{n}\notin (1+\varepsilon )\mathbb{E}%
P_{n}\}\leq 2\sum_{n=12}^{\infty }n^{-1-\varepsilon /4}=2\sum_{n=12}^{\infty
}\frac{1}{n(\log n)^{2}}<\infty 
\]%
Therefore, by the Borel-Cantelli lemma, with probability 1 there exists $%
N^{(1)}\in \mathbb{N}$ such that for all $n\geq N^{(1)}$,%
\begin{equation}
P_{n}\subseteq (1+8(\log \log n)/\log n)\mathbb{E}P_{n}
\label{outer sandwich}
\end{equation}%
For each $n\in \mathbb{N}$, let $E_{n}$ be the event that (\ref{outer
sandwich}) holds. Consider any sufficiently large (deterministic) $n\in 
\mathbb{N}$. Set $\varepsilon =3(\log \log n)/\log n$ and $\delta =3\exp
(-n^{-\varepsilon /2}/(6d))$. Let $\mathcal{N\subset \partial ((}\mathbb{E}%
P_{n})^{\circ })$ be a $\delta $-net as in Lemma \ref{net lemma}. As before, 
$\delta \leq \varepsilon /10\leq 1/20$. By the union bound and Lemma \ref{2
sided 1D concentration}, the following event, to be denoted $F_{n}$, occurs
with probability at least $1-(3/\delta )^{d}\exp (-n^{\varepsilon /2}/3)\geq
1-\exp (-n^{\varepsilon /2}/6)\geq 1-n^{-2}$: for all $\omega \in \mathcal{N}
$,%
\[
(1-\varepsilon /2)\left\Vert \omega \right\Vert _{(\mathbb{E}P_{n})^{\circ
}}\leq \left\Vert \omega \right\Vert _{P_{n}{}^{\circ }} 
\]%
The Borel-Cantelli lemma again implies that with probability $1$ there
exists $N^{(2)}\in \mathbb{N}$ such $F_{n}$ occurs for all $n\geq N^{(2)}$.
For all $n\geq \max \{N^{(1)},N^{(2)}\}$, $E_{n}\cap F_{n}$ occurs, and
expressing an arbitrary $\theta \in \mathcal{\partial ((}\mathbb{E}%
P_{n})^{\circ })$ as $\theta =\omega _{0}+\sum_{1}^{\infty }\delta
_{i}\omega _{i}$ as in Lemma \ref{net lemma} and using the triangle
inequality,%
\[
\left\Vert \theta \right\Vert _{P_{n}{}^{\circ }}\geq \left\Vert \omega
_{0}\right\Vert _{P_{n}{}^{\circ }}-\sum_{1}^{\infty }\delta ^{i}\left\Vert
\omega _{i}\right\Vert _{P_{n}{}^{\circ }}\geq 1-\varepsilon /2-2\delta
(1-\delta )^{-1}\geq 1-\varepsilon 
\]%
which implies (\ref{strong bound}).
\end{proof}

\section*{Acknowledgement}

The authors would like to thank Mokshay Madiman for comments related to the
paper.

\end{document}